\documentclass[11pt, reqno]{amsart}
\usepackage{amsmath, amsthm, amscd, amsfonts, amssymb, graphicx, color, mathrsfs}
\usepackage[bookmarksnumbered, colorlinks, plainpages=false]{hyperref}
\usepackage[all]{xy}
\usepackage[utf8]{inputenc}
\usepackage{slashed}
\usepackage{cleveref}
\usepackage{amsmath}
\usepackage{mathtools}
\usepackage{enumerate}
\usepackage{tgadventor}
\usepackage[
backend=biber,
style=alphabetic,
sorting=ynt
]{biblatex}
\newcommand{\beq}[1]{\begin{equation}\label{eq:#1}}
	\newcommand{\eeq}{\end{equation}}

\newtheorem{theorem}{Theorem}[section]
\newtheorem{lemma}[theorem]{Lemma}
\newtheorem{notation}[theorem]{Notation}

\newtheorem{proposition}[theorem]{Proposition}

\theoremstyle{definition}
\newtheorem{definition}[theorem]{Definition}
\newtheorem{example}[theorem]{Example}

\theoremstyle{remark}
\newtheorem{remark}[theorem]{Remark}
\numberwithin{equation}{section}

\def\G{{\mathbb{G}}}
\def\R{{\mathcal{R}}}
\def\L{{^{2}_{L^2(\G)}}}

\newcommand{\vertiii}[1]{{\left\vert\kern-0.25ex\left\vert\kern-0.25ex\left\vert #1 
    \right\vert\kern-0.25ex\right\vert\kern-0.25ex\right\vert}}
\allowdisplaybreaks

\begin{document}

	\setcounter{page}{1}
	
\title[Fractional hypoelliptic Schr\"odinger equations]{Fractional Schr\"odinger equations with singular potentials of higher order. II: Hypoelliptic case}

\author[M. Chatzakou]{Marianna Chatzakou}
\address{
	Marianna Chatzakou:
	\endgraf
    Department of Mathematics: Analysis, Logic and Discrete Mathematics
    \endgraf
    Ghent University, Belgium
  	\endgraf
	{\it E-mail address} {\rm marianna.chatzakou@ugent.be}
		}

\author[M. Ruzhansky]{Michael Ruzhansky}
\address{
  Michael Ruzhansky:
  \endgraf
  Department of Mathematics: Analysis, Logic and Discrete Mathematics
  \endgraf
  Ghent University, Belgium
  \endgraf
 and
  \endgraf
  School of Mathematical Sciences
  \endgraf
  Queen Mary University of London
  \endgraf
  United Kingdom
  \endgraf
  {\it E-mail address} {\rm michael.ruzhansky@ugent.be}
  }

\author[N. Tokmagambetov]{Niyaz Tokmagambetov}
\address{
  Niyaz Tokmagambetov:
  \endgraf
  Department of Mathematics: Analysis, Logic and Discrete Mathematics
  \endgraf
  Ghent University, Belgium
  \endgraf
  and
    \endgraf   
    Al--Farabi Kazakh National University
    \endgraf
    Almaty, Kazakhstan
    \endgraf
  {\it E-mail address} {\rm niyaz.tokmagambetov@ugent.be, tokmagambetov@math.kz}
  }

\thanks{The authors are supported by the FWO Odysseus 1 grant G.0H94.18N: Analysis and Partial Differential Equations and by the Methusalem programme of the Ghent University Special Research Fund (BOF) (Grant number 01M01021). Michael Ruzhansky is also supported by EPSRC grant EP/R003025/2. }

\begin{abstract} 
In this paper we consider the space-fractional Schr\"odinger equation with a singular potential for a wide class of fractional hypoelliptic operators. Such analysis can be conveniently realised in the setting of graded Lie groups. The paper is a continuation and extension of the first part \cite{ARST21b} where the classical Schr\"odinger equation on $\mathbb R^n$ with singular potentials was considered.  
\end{abstract}

\maketitle

\tableofcontents

\section{Introduction}
This paper is devoted to the fractional Schr\"odinger equation for positive (left) Rockland operator $\mathcal{R}$  (left-invariant hypoelliptic partial differential operator which is homogeneous of positive degree $\nu$) on a general graded Lie group $\mathbb{G}$, with a possibly singular potential; that is for $T>0$, and for $s>0$ we consider the Cauchy problem
\begin{equation}
\label{sch.eq}
\begin{cases}
iu_{t}(t,x) +\mathcal{R}^{s}u(t,x)+p(x)u(t,x)=0\,,\quad (t,x)\in [0,T]\times \mathbb{G}\,,\\
u(0,x)=u_0(x)\,,u_{t}(0,x)=u_1(x),\; x \in \G\,,	
\end{cases}       
\end{equation}
where $p$ is a non-negative distributional function.

The main idea of this paper is to relax the regularity assumptions on the potential $p$ in \eqref{sch.eq}. The coefficient $p$ is allowed to have $\delta$-function type singularities. But the question of having a suitable notion of solutions to \eqref{sch.eq} is still open. The situation is reaching an impasse by the well--known impossibility problem on the multiplication of distributions stated by Schwartz in \cite{Sch54}. To deal with it, in this paper we use the concept of very weak solutions introduced in \cite{GR15} to work with the wave equations with irregular coefficients. Later, the developed tools were applied to other equations with singular coefficients \cite{RT17a}, \cite{RT17b}, \cite{ART19}, and \cite{MRT19}. 
In all these papers authors work with the time-dependent equations and in the recent works \cite{ARST21a}, \cite{ARST21b}, 
 \cite{ARST21c} and \cite{Gar20} one started a development of the very weak solutions for partial differential equations with (strongly singular) space-depending coefficients. 

In this paper, we will show in details that the notion of very weak solutions is applicable to the Cauchy problem \eqref{sch.eq} for the fractional Schr\"odinger equation for the Rockland operator $\mathcal{R}$ on the graded Lie group $\mathbb{G}$ with a strongly singular potential-coefficient depending on the spacial variable. Indeed, the present work is an improvement and extension of the results obtained in the first part \cite{ARST21b} addressed to the fractional Schr\"odinger equation. It should be mentioned that the setting of \cite{ARST21b} was the equation \eqref{sch.eq} for $\G=\mathbb R^d$ and $\mathcal{R}=(-\Delta)^{s}$ being the positive fractional Laplacian on the Euclidean space. Consequently, the results of \cite{ARST21b} can be considered as a special case of the results obtained here.

At the same time, we also give here some corrections and clarifications to statements of the first part \cite{ARST21b}, see Remark \ref{rem.negl} and Remark \ref{finrem}. Also, the arguments around the Sobolev embedding techniques starting from Proposition \ref{prop2} are new here, giving a new result also for the setting in \cite{ARST21b}.

Let us briefly recall the necessary notions in the setting of graded groups. For a more detailed exposition we refer to Folland and Stein [Chapter 1 in \cite{FS82}], or to the more recent one by Fischer and the second author [Chapter 3 in \cite{FR16}]. 
\par A connected simply connected Lie group $\G$ is called a graded Lie group if its Lie algebra $\mathfrak{g}$ has a vector space decomposition of the following form 
\[ 
 \mathfrak{g}= \bigoplus_{i=1}^{\infty}  \mathfrak{g}_{i}\,,
\]
where all, but finitely many $\mathfrak{g}_{i}$'s, are equal to $\{0\}$, and we have 
\[
[\mathfrak{g}_{i},\mathfrak{g}_{j}]\subset \mathfrak{g}_{i+j}\footnote{For the vector spaces $V$, $W$ we denote by $[V,W]$ the vector space $\{[v,w]: v \in V\,,w \in W\}$, where $[v,w]:=vw-wv$ is the Lie bracket.}\,,\quad \text{for all}\quad i,j \in \mathbb{N}\,.
\]
For a graded Lie group $\G \sim \mathbb{R}^n$ with Lie algebra $\mathfrak{g}$, we fix a basis $\{X_1,\cdots,X_n\}$ of $\mathfrak{g}$ adapted to the gradation of the above form. Then, the exponential map $\exp_{\G}: \mathfrak{g} \rightarrow \G$ defined as
\[
x:=\exp_{\G}(x_1X_1+\cdots+x_nX_n)\,,
\]
is a global diffeomorphism from $\mathfrak{g}$ onto $\G$.
\par Let $A$ be a diagonalisable linear operator on $\mathfrak{g}$ with positive eigenvalues. Then, a family $\{D_r\}_{r>0}$ of dilations of $\mathfrak{g}$ is a collection of linear mappings of the form 
\[
D_r=\textnormal{Exp}(A\, \textnormal{ln}r)=\sum_{k=0}^{\infty}\frac{1}{k!}(\textnormal{ln}(r)A)^k\,,
\]
where $\textnormal{Exp}$ denotes the exponential of matrices. Moreover, the exponential mapping $\textnormal{exp}_\G$ on $\G$ transports the dilations $\{D_r\}_{r>0}$ on the group side; i.e., we have 
\begin{equation}\label{dw}
D_r(x)=rx=(r^{\nu_{1}}x_1,\cdots, r^{\nu_{n}}x_n)\,,x \in \G\,,
\end{equation}
where $\nu_{1},\cdots,\nu_{n}$ are the weights of the dilations. Additionally, each $D_r$, $r>0$, is a morphism of $\mathfrak{g}$, and consequently, also of $\G$. 
\par Finally, let us note that a connected simply connected Lie group $\G$ that can be equipped with such a family of automorphisms is called a homogeneous Lie group; for such groups, the quantity 
\[
Q:= \textnormal{Tr}A=\nu_{1}+\cdots+\nu_{n}\,,
\]
is called the homogeneous dimension of $\G$.\\

Recall that graded Lie groups are naturally also homogeneous Lie groups. Let us illustrate the above ideas with the following examples in the settings of two well-studied homogeneous Lie groups.
\begin{example}\label{1exa}
\textit{Heisenberg group $\mathbb{H}_n$, $n \in \mathbb{N}$:} Let $\mathfrak{h}_n$ be the Lie algebra of the Heisenberg group $\mathbb{H}_n \sim \mathbb{R}^{2n+1}$ with elements  \[\mathfrak{h}_n=\{X_1,\cdots,X_n,Y_1,\cdots,Y_n,T\}\,,\] that satisfy the (non-zero) commutator relations
 		\[
 		[X_i,Y_i]=T\,,\quad i=1,\cdots,n\,.
 		\]
 		Therefore, $\mathfrak{h}_n$ admits the following gradation
 		\[
 		\mathfrak{h}_n=V_1 \oplus V_2=\textnormal{span}\{X_1,\cdots,X_n,Y_1,\cdots,Y_n\}\oplus \mathbb{R}T\,.
 		\] 
 		The dilations on the elements of $\mathfrak{h}_n$ are given by 
 		\[
 		D_r(X_i)=rX_i\,,D_r(Y_i)=rY_i\,,i=1,\cdots,n\,,D_r(T)=r^2 T\,,
 		\]
 		and subsequently $\mathbb{H}_n$ is homogeneous when equipped with the dilations
 		\[
 		rh=(rx,ry,r^2 t)\,,\quad h=(x,y,t)\in \mathbb{R}^n \times \mathbb{R}^n \times \mathbb{R}\,,
 		\] 
 		with homogeneous dimension  $Q_{\mathbb{H}_n}=n+n+2=2n+2$.
 \end{example}
 \begin{example}\label{1exa-2}
 \textit{Engel group $\mathcal{B}_4$:} Let $\mathfrak{l}_4$ be the Lie algebra of the Engel group $\mathcal{B}_4 \sim \mathbb{R}^4$ with elements
 		\[
 		\mathfrak{l}_4=\{X_1,X_2,X_3,X_4\}\,,
 		\]
 		that satisfy the (non-zero) commutator relations
 		\[
 		[X_1,X_2]=X_3\,,\quad [X_1,X_3]=X_4\,.
 		\]
 		The Lie algebra $\mathfrak{l}_4$ admits the gradation
 		\[
 		\mathfrak{l}_4=V_1 \oplus V_2 \oplus V_3=\textnormal{span}\{X_1,X_2\}\oplus \mathbb{R}X_3 \oplus \mathbb{R}X_4\,,
 		\]
 		and the natural dilations on $\mathfrak{l}_4$ are given by 
 		\[
 		D_r(X_1)=rX_1\,,D_r(X_2)=rX_2\,,D_r(X_3)=r^2X_3\,,D_r(X_4)=r^3X_4\,,
 		\]
 		which, transported to the group side, yield 
 		\[
 		rx=(rx_1,rx_2,r^2x_3,r^3x_4)\,, x=(x_1,x_2,x_3,x_4)\in \mathcal{B}_4\,.
 		\]
 		Its homogeneous dimension is $Q_{\mathcal{B}_4}=1+1+2+3=7$.
 \end{example}
 \par Let $\pi$ be a representation of the group $\G$ on the separable Hilbert space $\mathcal{H}_\pi$. We say that a vector $v \in \mathcal{H}_\pi$ is a smooth vector and we write $v \in \mathcal{H}_{\pi}^{\infty}$, if the function 
 \[
 \G \ni x \mapsto \pi(x)v \in \mathcal{H}_\pi\,,
 \]
 is of class $C^\infty$. If $X \in \mathfrak{g}$, then for $v \in \mathcal{H}_{\pi}^{\infty}$, the limit 
 \[
 d\pi(X)v:= \lim_{t \rightarrow 0} \frac{1}{t}\left(\pi(\exp_\G(tX))v-v \right)\,,
 \]
 exists, and the mapping $d\pi: \mathfrak{g}\rightarrow \textnormal{End}(\mathcal{H}_{\pi}^{\infty})$ is the infinitesimal representation of $\mathfrak{g}$ on $\mathcal{H}_{\pi}^{\infty}$ associated to $\pi$. With an abuse of notation we shall write $\pi$ for the infinitesimal representation of $\mathfrak{g}$. Setting $\pi(X)^{\alpha}=\pi(X^{\alpha})$, $\alpha \in \mathbb{N}^n$, we can extend the infinitesimal representation of $\mathfrak{g}$ to elements of the universal enveloping algebra $\mathfrak{U}(\mathfrak{g})$; i.e., we can write 
 \[
 d\pi(T):=\pi(T)\,,\quad T \in \mathfrak{U}(\mathfrak{g})\,,
 \]
 where $T$ has been identified with a left-invariant operator on $\G$, and the set of infinitesimal representations $\{\pi(T): \pi \in \widehat{\G}\}$ is a fields of operators that turns out to be the symbol associated to $T$.
 \par Recall that an immediate consequence of the so-called \textit{Poincar\'{e}-Birkhoff-Witt} Theorem is that $\mathfrak{U}(\mathfrak{g})$ can be identified with the space of left-invariant differential operators on $\G$, and moreover, any left-invariant vector field can be written in a unique way as the sum 
	\begin{equation}\label{pbw}
	\sum_{\alpha \in \mathbb{N}^n}c_{\alpha}X^{\alpha}\,,
\end{equation}
	where all but finite $c_{\alpha} \in \mathbb{C}$ are zero, and where for $X_j \in \mathfrak{g}$ we have defined $X^{\alpha}:=X_{1}^{\alpha_{1}}\cdots X_{n}^{\alpha_{n}}$, for $\alpha=(\alpha_{1},\cdots,\alpha_{n}) \in \mathbb{N}^n$. 
 \par If $\pi \in \widehat{\G}$; that is, if $\pi$ is an element of the unitary dual of $\G$, then we say that the left-invariant differential operator $\R$ on $\G$, which is homogeneous of a positive degree, is a Rockland operator, if it satisfies the following Rockland condition:
 \[
 (\textbf{R})\,\text{for every non-trivial representation}\, \pi \in \widehat{\G}\,  \text{the operator}\,\pi(\R)\, \text{is injective on}\,\mathcal{H}_{\pi}^{\infty}\,,\text{i.e.,}
 \]
 \[
 \forall v \in \mathcal{H}_{\pi}^{\infty}\,,\pi(\R)v=0\Longrightarrow v=0\,.
 \]
 For a more detailed discussion of the above condition as appeared in the work of Rockland \cite{Roc78} we refer to [Sections 1.7 and 4.1 in \cite{FR16}]. Equivalent to $(\textbf{R})$ conditions have appeared in the works of Beals \cite{Bea77} and Helffer and Nourrigat \cite{HN79}, with the latter characterising Rockland operators as being the left-invariant hypoelliptic differential operators on $\G$. Spectral properties of the infinitesimal representations of Rockland operators have been considered in \cite{TR97}.
   
\par 
For $\pi \in \widehat{\G}$, and for $\R$ being a positive Rockland operator of homogeneous degree $\nu$, from \eqref{pbw} we obtain the following representation of symbol associated to $\R$,
\[
\pi(\R)=\sum_{[\alpha]=\nu}c_\alpha\pi(X)^{\alpha}\,,
\]
where $$[\alpha]=\nu_1 \alpha_1+\cdots+\nu_n \alpha_n$$ is the homogeneous length of the multi-index $\alpha$, and $$\pi(X)^{\alpha}=\pi(X^\alpha)=\pi(X_{1}^{\alpha_1}\cdots X_{n}^{\alpha_n}),$$ 
where $X_j$ is of homogeneous degree $\nu_j$.
\par Recall that $\R$ and $\pi(\R)$ are densely defined on $\mathcal{D}(\G)\subset L^2(\G)$ and  on $\mathcal{H}_{\pi}^{\infty}\subset \mathcal{H}_\pi$, respectively (see, e.g. [Proposition 4.1.15 in \cite{FR16}]). Additionally, let us mention that, for the groups we consider here, in the case where $\mathcal{H}_\pi=L^2(\mathbb{R}^m)$ we have $\mathcal{H}_{\pi}^{\infty}=\mathcal{S}(\mathbb{R}^m)$, see [Corollary 4.1.2 in \cite{CG90}]. From now on, let us denote by $\R$, and by $\pi(\R)$ the self-adjoint extensions of the above on the spaces $L^2(\G)$, and $\mathcal{H}_\pi$, respectively.
\par By the spectral theorem for unbounded operators (see, e.g. Theorem VIII.6 in \cite{RS85}) we can write
\[
\R=\int_{\mathbb{R}}\lambda\,dE(\lambda)\,,\quad \text{and}\quad \pi(\R)=\int_{\mathbb{R}}\lambda\,dE_{\pi}(\lambda)\,,
\]
where $E$ and $E_\pi$ stand for the spectral measures associated to $\R$ and to $\pi(\R)$. 
\par For our purposes, we have required the positivity of the Rockland operator $\mathcal{R}$ that should be regarded in the operator sense. In particular, the Rockland operator $\R$ is positive on $L^2(\G)$, if it is formally self-adjoint; that is we have $\mathcal{R}=\mathcal{R}^{*}$ in the universal enveloping algebra $\mathfrak{U}(\mathfrak{g})$, and $\R$ satisfies the condition
\[
\int_{\G}\mathcal{R}f(x)\overline{f(x)}\,dx \geq 0\,,\quad \forall f \in \mathcal{D}(\G)\,.
\] 
For a positive Rockland operator $\R$, the infinitesimal representations $\pi(\R)$ are also positive because of the relations between the spectral measures.
\par A standard example of a Rockland operator on a stratified Lie group $\G$ is the so-called sub-Laplacian on $\G$ that is of homogeneous degree $\nu=2$, and is defined as follows:
\par If $\G$ is a stratified Lie group with a given basis $Z_1,\cdots,Z_k$ for the first stratum of its Lie algebra, then the left-invariant differential operator on $\G$
given by
\[
Z_{1}^{2}+\cdots+Z_{k}^{2}\,,
\]
is called the sub-Laplacian on $\G$. 
\par The infinitesimal representations of such operators on the particular cases of the Heisenberg and Engel groups, as introduced in Examples \ref{1exa} and \ref{1exa-2}, are given in the following examples.
\begin{example}
Heisenberg group $\mathbb{H}_n$ : Using the Schr\"odinger representations (see e.g. \cite{Tay84}) of $\mathbb{H}_n$, the inifinitesimal representation, parametrised by $\lambda \in \mathbb{R}\setminus \{0\}$, of the sub-Laplacian $\mathcal{L}_{\mathbb{H}}$ on $\mathbb{H}_n$, is the operator on $\mathcal{H}_{\pi_\lambda}^{\infty}= \mathcal{S}(\mathbb{R}^{n})$ given by
    \[
   \mathcal{A}:= \pi_\lambda(\mathcal{L}_{\mathbb{H}})=|\lambda| \sum_{j=1}^{n} \left(\partial^{2}_{u_j}-u_{j}^{2} \right)\,.
    \]
    Now, $\mathcal{A}$ is the harmonic oscillator, and if we keep the same notation for its self-adjoint extension on $\mathcal{H}_{\pi_\lambda}=L^2(\mathbb{R}^n)$, then the spectrum of $-\mathcal{A}$ is explicitly known as
  \[
    \{ 2|\ell|+n\,,\ell \in \mathbb{N}^n\}\,,
    \]
    where $|\ell|=\ell_1+\cdots+\ell_n$, see, e.g. [Section 6.4 in \cite{FR16}].
\end{example}

\begin{example}
Engel group $\mathcal{B}_4$: Using the representations of $\mathcal{B}_4$ proved by Dixmier [p.333 in \cite{Dix57}] we see that the infinitesimal representation, parametrised by $\lambda \in \mathbb{R}\setminus \{0\}\,,\mu \in \mathbb{R}$, of the sub-Laplacian $\mathcal{L}_{\mathcal{B}_4}$ on $\mathcal{B}_4$, is the operator on $\mathcal{H}_{\pi_{\lambda,\mu}}=\mathcal{S}(\mathbb{R})$ given by 
    \[
    \mathcal{A}:=\pi_{\lambda,\mu}(\mathcal{L}_{\mathcal{B}_4})=\frac{d^2}{du^2}-\frac{1}{4}\left(\lambda u^2-\frac{\mu}{\lambda} \right)^2\,.
    \]
    The operator $\mathcal{A}$ here is an anharmonic oscillator that admits a self-adjoint extension on $\mathcal{H}_{\pi_\lambda}=L^2(\mathbb{R}^n)$ and has discrete spectrum, see e.g. \cite{CDR18}.
\end{example}
More generally, for a positive Rockland operator $\R$, Hulanicki, Jenkins and Ludwig \cite{HJL85} proved that the spectrum of $\pi(\R)$, with $\pi \in \widehat{\G}\setminus \{1\}$, is discrete and lies in $(0,\infty)$, which allows us to choose an orthonormal basis for $\mathcal{H}_\pi$ such that the self-adjoint operator $\pi(\R)$ admits an infinite matrix representation of the form 
\begin{equation}\label{repr.pr}
\pi(\R)=\begin{pmatrix}
\pi_{1}^{2} & 0 & \cdots & \cdots\\
0 & \pi_{2}^{2} & 0 & \cdots\\
\vdots & 0 & \ddots & \\
\vdots & \vdots & & \ddots
\end{pmatrix}\,,
\end{equation}
where $\pi \in \widehat{\G}\setminus \{1\}$ and $\pi_j>0$.
 \par We will now briefly recall the group Fourier transform: If we identify the irreducible unitary representations with their equivalence classes, then for $f \in L^1(\G)$ and for $\pi \in \widehat{\G}$, the group Fourier transform of $f$ at $\pi$ is the map
 \[
 \mathcal{F}_{\G}f: \pi \mapsto \mathcal{F}_{\G}f(\pi)\,,
 \]
 that is a linear endomorphism on $\mathcal{H}_\pi$, defined by 
 \[
 \mathcal{F}_{\G}f(\pi)\equiv \widehat{f}(\pi) \equiv \pi(f):= \int_{\G}f(x)\pi(x)^{*}\,dx\,,
 \]
 where the integration on $\G$ is taken with respect to the binvariant Haar measure $dx$ on $\G$. By the above we can also write
 \[
 \mathcal{F}_{\G}(\R f)(\pi)=\pi(\R)\widehat{f}(\pi)\,,
 \]
 and, using the basis in the representation of $\mathcal{H}_\pi$ given in \eqref{repr.pr}, the latter can be rewritten as
 \[
\left\{ \pi_{k}^{2}\cdot \widehat{f}(\pi)_{k,l}\right\}_{k,l \in \mathbb{N}}\,.
 \]
 \par For graded Lie groups, or more generally for connected simply connected nilpotent Lie groups, the orbit method or, more particularly, the geometry of co-adjoint orbits \cite{CG90, Kir04}, identifies the unitary dual $\widehat{{\G}}$ with a subset of a Euclidean space which is equipped with a concrete measure $\mu$, called the Plancherel measure, that allows for the Fourier inversion formula. Furthermore, the operator $\pi(f)$ is in the Hilbert-Schmidt class, and its Hilbert-Schmidt norm depends only on the class of $\pi$; the map \[\widehat{{\G}} \ni \pi \mapsto \|\pi(f)\|^{2}_{\textnormal{HS}}\] is integrable against $\mu$ and we have the following isometry, known as the Plancherel formula%
		\begin{equation}\label{planc.id}
		\int_{\G} |f(x)|^2\,dx=\int_{\widehat{{\G}}}\|\pi(f)\|^{2}_{\textnormal{HS}}\,d\mu(\pi)\,.
		\end{equation}
For a detailed discussion on this topic we refer to [Section 1.8, Appendix
B.2 in \cite{FR16}].
\par Finally, since the action of a Rockland operator $\R$ is involved in our analysis, let us make a brief overview of some related properties.
\begin{definition}[Homogeneous Sobolev spaces]
For $s>0$, $p>1$, and $\R$ a positive homogeneous Rockland operator of degree $\nu$, we define the $\R$-Sobolev spaces as the space of tempered distributions $\mathcal{S}^{'}(\G)$ obtained by the completion of $\mathcal{S}(\G)\cap \textnormal{Dom}(\R^{\frac{s}{\nu}})$ for the norm
\[
\|f\|_{\dot{L}^{p}_{s}(\G)}:=\|\R^{\frac{s}{\nu}}_{p}f\|_{L^p(\G)}\,,\quad f \in \mathcal{S}(\G)\cap \textnormal{Dom}(\R^{\frac{s}{\nu}}_{p})\,,
\]
where $\R_p$ is the maximal restriction of $\R$ to $L^{p}(\G)$.\footnote{When $p=2$, we will write $\R_2=\R$ for the self-adjoint extension of $\R$ on $L^2(\G)$.}
\end{definition}
Let us mention that, the above $\R$-Sobolev spaces do not depend on the specific choice of $\R$, in the sense that, different choices of the latter produce equivalent norms, see [Proposition 4.4.20 in \cite{FR16}]. 
\par In the scale of these Sobolev spaces, we recall the next proposition as in [Proposition 4.4.13 in \cite{FR16}].
\begin{proposition}[Sobolev embeddings]
For $1<\tilde{q}_0<q_0<\infty$ and for $a,b \in \mathbb{R}$ such that 
\[
b-a=Q \left( \frac{1}{\tilde{q}_0}-\frac{1}{q_0}\right)\,,
\]
we have the continuous inclusions 
\[
\dot{L}^{\tilde{q}_0}_{b}(\G) \subset \dot{L}^{q_0}_{a}(\G)\,,
\]
that is, for every 
$f \in \dot{L}^{\tilde{q}_0}_{b}(\G)$, we have $f \in \dot{L}^{q_0}_{a}(\G)$, and there exists some positive constant $C=C(\tilde{q}_0,q_0,a,b)$ (independent of $f$) such that 
\begin{equation}
\label{inclusions}
\|f\|_{\dot{L}^{q_0}_{a}(\G)}\leq C \|f\|_{\dot{L}^{\tilde{q}_0}_{b}(\G)}\,.
\end{equation}
\end{proposition}
\par In the sequel we will make use of the following notation:
\begin{notation}
\begin{itemize}
    \item When we write $a \lesssim b$, we will mean that there exists some constant $c>0$ (independent of any involved parameter) such that $a \leq c b$;
    \item if $\alpha=(\alpha_1,\cdots,\alpha_n) \in \mathbb{N}^n$ is some multi-index, then we denote by 
    \[
    [\alpha]=\sum_{i=1}^{n} v_i\alpha_i\,,
    \]
    its \textit{homogeneous length}, where the $v_i$'s stand for the dilations' weights as in \eqref{dw}, and by 
    \[
    |\alpha|=\sum_{i=1}^{n}\alpha_i\,,
    \]
    the length of it;
    \item for suitable $f \in \mathcal{S}^{'}(\G)$  we have introduced the following norm
    \[
    \|f\|_{H^{s}(\G)}:=\|f\|_{\dot{L}^{2}_{s}(\G)}+\|f\|_{L^2(\G)}\,;
    \]
    \item when regulisations of functions/distributions on $\G$ are considered, they must be regarded as arising  via  convolution  with  Friedrichs-mollifiers;  that is, $\psi$ is a Friedrichs-mollifier, if it is a compactly supported smooth function with $\int_{\G}\psi\,dx=1$. Then the regularising net is defined as
\begin{equation}\label{mol}
	\psi_{\epsilon}(x)=\epsilon^{-Q}\psi(D_{\epsilon^{-1}}(x))\,,\quad \epsilon \in (0,1]\,,
\end{equation}
where $Q$ is the homogeneous dimension of $\G$.
\end{itemize}
\end{notation}

\section{Estimates for the classical solution}

Here and thereafter, we consider a fixed power $s>0$ of a fixed, positive  Rockland operator $\R$ that is assumed to be of homogeneous degree $\nu$. Moreover, the coefficient $p$ in \eqref{sch.eq} will be regarded to be non-negative on $\G$.
\par The next two propositions prove the existence and uniqueness of the classical solution to the Cauchy problem \eqref{sch.eq}, in the cases where the potential $p$ is in the space $ L^{\infty}(\G)$ or $ L^{\frac{2Q}{\nu s}}(\G)$, where, in the second case, the condition $Q> \nu s$ must be satisfied.
\begin{proposition}\label{prop1.clas.schr}
Let $p \in L^{\infty}(\G)$, where $p \geq 0$, and suppose that $u_0 \in H^{\frac{s \nu}{2}}(\G)$. Then, there exists a unique solution $u \in C([0,T];H^{\frac{s \nu}{2}}(\G))$ to the Cauchy problem \eqref{sch.eq}, that satisfies the estimate 
	\begin{equation}\label{prop1.claim}
		\|u(t,\cdot)\|_{H^{\frac{s \nu}{2}}(\G)} \lesssim (1+\|p\|_{L^{\infty}(\G)}) \|u_0\|_{H^{\frac{s \nu}{2}}(\G)}\,,
		\end{equation}
		uniformly in $t \in [0,T]$.
\end{proposition}
\begin{proof}
Multiplying the equation \eqref{sch.eq} by $u_t$ and integrating over $\mathbb{G}$, we get
	\begin{equation}\label{Re0,sch}
		\Re(\langle iu_{t}(t,\cdot),u_t(t,\cdot)\rangle_{L^2(\mathbb{G})}+\langle \mathcal{R}^{s}u(t,\cdot),u_t(t,\cdot)\rangle_{L^2(\G)}+\langle p(\cdot)u(t,\cdot),u_t(t,\cdot) \rangle_{L^2(\G)})=0\,,
	\end{equation}
for all $t \in [0,T]$. It is easy to see that 

\[
\Re(\langle \mathcal{R}^{s}u(t,\cdot),u_t(t,\cdot)\rangle_{L^2(\G)})=\frac{1}{2} \partial_{t}\| \R^{\frac{s}{2}}u(t,\cdot)\|_{L^2(\G)}^{2}\,,
\]
and
\[
\Re(\langle p(\cdot)u(t,\cdot),u_t(t,\cdot) \rangle_{L^2(\G)})=\frac{1}{2} \partial_{t} \| \sqrt{p}(\cdot)u(t,\cdot)\|_{L^2(\G)}^{2}\,,
\]
so that, denoting by  $$E(t):=\|\R^{\frac{s}{2}}u(t,\cdot)\|\L+\|\sqrt{p}(\cdot)u(t,\cdot)\|\L,$$ the real part of the functional estimate of \eqref{Re0,sch}, equation \eqref{Re0,sch} implies that $\partial_{t}E(t)=0$, and consequently also that \begin{equation}\label{et=e0}E(t)=E(0)\,,\quad \text{for all}\quad t \in [0,T]\,.
\end{equation}
Therefore, taking into consideration the estimate
\[
\|\sqrt{p}u_0\|\L \leq \|p\|_{L^{\infty}(\G)} \|u_0\|\L\,,
\]
the equation \eqref{et=e0} implies that for all $t \in [0,T]$ we have 
\begin{equation}\label{pu.1.est}
\|\sqrt{p}u(t,\cdot)\|\L \lesssim \|\R^{\frac{s}{2}}u_0\|\L+\|p\|_{L^{\infty}}\|u_0\|\L\,,
\end{equation}
and 
\begin{equation}\label{Ru.1.est}
\|\R^{\frac{s}{2}}u(t,\cdot)\|\L \lesssim \|\R^{\frac{s}{2}}u_0\|\L+\|p\|_{L^{\infty}}\|u_0\|\L\,.
\end{equation}
Now since 
\[
\|\R^{\frac{s}{2}}u_0\|\L\,, \|u_0\|\L \leq \|u_0\|^{2}_{H^{\frac{s \nu}{2}}(\G)}\,,
\]
we can estimate \eqref{pu.1.est} and \eqref{Ru.1.est} further by 
\begin{equation}\label{pu.2.est}
    \|\sqrt{p}u(t,\cdot)\|_{L^2(\G)} \lesssim \left( 1+\|p\|^{\frac{1}{2}}_{L^{\infty}(\G)}\right) \|u_0\|_{H^{\frac{s \nu}{2}}(\G)}\,,
\end{equation}
and
\begin{equation}\label{Ru.2.est}
    \|\R^{\frac{s}{2}}u(t,\cdot)\|_{L^2(\G)}\lesssim \left( 1+\|p\|^{\frac{1}{2}}_{L^{\infty}(\G)}\right) \|u_0\|_{H^{\frac{s \nu}{2}}(\G)}\,,
\end{equation}
respectively.

Now, to prove \eqref{prop1.claim}, it remains to show the desired estimate for the norm $\|u(t,\cdot)\|_{L^2(\G)}$. To this end, we first apply the group Fourier transform to \eqref{sch.eq} with respect to $x \in \G$ and for all $\pi \in \widehat{\G}$, and we get 
\begin{equation}
	\label{prop1.ft}
	i\widehat{u}_{t}(t,\pi)+\pi(\R)^s\, \widehat{u}(t,\pi)=\widehat{f}(t,\pi);\quad \widehat{u}(0,\pi)_{k,l}=\widehat{u}_0(\pi)_{k,l}\,,
\end{equation}
where $\widehat{f}(t,\pi)$ denotes the group Fourier transform of the function $f(t,x):=-p(x)u(t,x)$. Taking into account the matrix representation of $\pi(\R)$, we rewrite the matrix equation \eqref{prop1.ft} componentwise as the infinite system of equations of the form 
\begin{equation}
	\label{prop1.ft2}
	i\widehat{u}_{t}(t,\pi)_{k,l}+\pi^{2s}_{k}\cdot \widehat{u}(t,\pi)_{k,l}=\widehat{f}(t,\pi)_{k,l}\,,
\end{equation}
 for all $\pi \in \widehat{\G}$ and for any $k,l \in \mathbb{N}$, where now $\widehat{f}(t,\pi)_{k,l}$ can be regarded as the source term of the second order differential equation as in \eqref{prop1.ft2}.
\par  Now, let us decouple the matrix equation in \eqref{prop1.ft2} by fixing $\pi \in \widehat{\G}$, and treat each of the equations represented in \eqref{prop1.ft2} individually. If we denote by 
\[
v(t):=\widehat{u}(t,\pi)_{k,l}\,,\beta^{2s}:=\pi^{2s}_{k}\,,f(t):=\widehat{f}(t,\pi)_{k,l}\quad \text{and}\quad v_0:=\widehat{u}_0(\pi)_{k,l}\,,
\]  
then \eqref{prop1.ft2} becomes
\begin{equation}\label{prop1.fix.pi}
	iv^{'}(t)+\beta^{2s}\cdot v(t)=f(t);\,
	v(0)=v_0\,,
\end{equation}
with $\beta>0.$
By solving first the homogeneous version of \eqref{prop1.fix.pi}, and then by applying Duhamel's principle (see e.g. \cite{Eva98}), we get the following representation of the solution of \eqref{prop1.fix.pi}
\[
v(t)=v_0\,\textnormal{exp}(-i \beta^{2s}t)+\int_{0}^{t}\textnormal{exp}(-i \beta^{2s}(t-s))f(s)\,ds\,.
\]
Therefore, if we substitute back our initial conditions in $t$, then we get the estimate
\begin{equation}\label{b.HS}
|\widehat{u}(t,\pi)_{k,l}|^2 \lesssim |\widehat{u}_0(\pi)_{k,l}|^2+\int_{0}^{T}|\widehat{f}(t,\pi)_{k,l}|^2\,dt\,,
\end{equation}
which holds uniformly in $\pi \in \widehat{\G}$ and for each $k,l \in \mathbb{N}$, where we have used that $L^2([0,T])\subset L^1([0,T])$.
Now, recall that since for any Hilbert-Schmidt operator $A$ one has 
\[
\|A\|^{2}_{\textnormal{HS}}=\sum_{k,l}|\langle A \varphi_{k},\varphi_{l}\rangle|^2\,,
\]
where $\{\varphi_1,\varphi_2,\cdots\}$ is some orthonormal basis, summing over $k,l \in \mathbb{N}$ the inequalities  \eqref{b.HS} we get 
\[
\|\widehat{u}(t,\pi)\|^{2}_{\textnormal{HS}}\lesssim \|\widehat{u}_0(\pi)\|^{2}_{\textnormal{HS}}+\sum_{k,l}\int_{0}^{T}|\widehat{f}(t,\pi)_{k,l}|^2\,dt\,.
\]
Next we integrate the last inequality with respect to the Plancherel measure $\mu$ on $\widehat{\G}$, so that using the Plancherel identity \eqref{planc.id}, we obtain 
\begin{equation}\label{before.Pl}
	\|u(t,\cdot)\|\L \lesssim  \|u_0\|\L+\int_{\G} \sum_{k,l}\int_{0}^{T}|\widehat{f}(t,\pi)_{k,l}|^2\,dt\,d\mu(\pi)\,,
\end{equation}
and if we use Lebesgue's dominated convergence theorem, Fubini's theorem and the Plancherel formula we have
\begin{equation}\label{F,dc,Pl}
\int_{\G} \sum_{k,l}\int_{0}^{T}|\widehat{f}(t,\pi)_{k,l}|^2\,dt\,d\mu=\int_{0}^{T}\int_{\G}\sum_{k,l}|\widehat{f}(t,\pi)_{k,l}|^2\,d\mu\,dt=\int_{0}^{T}\|f(t,\cdot)\|_{L^2(\G)}^{2}\,dt\,.
\end{equation}
Now, since $f(t,x)=-p(x)u(t,x)$, using the estimate \eqref{pu.2.est} we get
\begin{equation}\label{est.f.p}
 \|f(t,\cdot)\|_{L^2(\G)}\leq \|p\|^{\frac{1}{2}}_{L^{\infty}(\G)}\|\sqrt{p}u(t,\cdot)\|_{L^2(\G)}\lesssim \left(1+\|p\|_{L^{\infty}(\G)} \right)\|u_0\|_{H^{\frac{s \nu}{2}}(\G)}\,,
\end{equation}
so that by \eqref{before.Pl} we arrive at 
\begin{equation}
    \label{prop1.est.u}
    \|u(t,\cdot)\|_{L^2(\G)}\lesssim \left(1+\|p\|_{L^{\infty}(\G)} \right)\|u_0\|_{H^{\frac{s \nu}{2}}(\G)}\,.
\end{equation}

Finally, combining the inequalities \eqref{Ru.2.est} and \eqref{prop1.est.u} we get 
\begin{eqnarray}\label{prop1.est.B}
	\|u(t,\cdot)\|_{L^2(\G)} & \lesssim &  (1+\|p\|_{L^\infty(\G)})\|u_0\|_{H^{\frac{s \nu}{2}}(\G)}\,,
	\end{eqnarray}
uniformly in $t \in [0,T]$, and this shows the estimate \eqref{prop1.claim} while the uniqueness of $u$ also follows. This completes the proof of Proposition \ref{prop1.clas.schr}. 
\end{proof}
\begin{proposition}\label{prop2}
Assume that $Q >\nu s$, and let $p \in L^{\frac{2Q}{\nu s}}(\G)\cap L^{\frac{Q}{\nu s}}(\G)$, $p \geq 0$. If we suppose that $u_0 \in H^{\frac{s \nu}{2}}(\G)$, then there exists a unique solution $u \in C([0,T];H^{\frac{s \nu}{2}}(\G))$ to the Cauchy problem \eqref{sch.eq} satisfying the estimate 
\begin{equation}
\label{prop2.claim}
\|u(t,\cdot)\|_{H^{\frac{s \nu}{2}}(\G)}\lesssim \|u_0\|_{H^{\frac{s \nu}{2}}(\G)}  \left\{\left(1+\|p\|_{L^{\frac{2Q}{\nu s}}(\G)}\right) \left(1+\|p\|_{L^{\frac{Q}{\nu s}}(\G)}\right)^{\frac{1}{2}}\right\}\,,
\end{equation}
 uniformly in $t \in [0,T]$.
\end{proposition}
\begin{proof}
Proceeding as in the proof of Proposition \ref{prop1.clas.schr}, we have 
\begin{equation}\label{prop2.et}
E(t)=E(0)\,,\quad \forall t \in [0,T]\,,
\end{equation}
where the energy estimate $E$ is given by 
\[
E(t)=\|\R^{\frac{s}{2}}u(t,\cdot)\|\L+\|\sqrt{p}(\cdot)u(t,\cdot)\|\L\,.
\]
Now, applying H\"older's inequality, we get 
\begin{equation}
\label{prop2.Hold}
 \|\sqrt{p} u_0\|\L \leq \|p\|_{L^{q^{'}}(\G)}\|u_{0}\|^{2}_{L^{2q}(\G)},
\end{equation}
where $1<q,q^{'}<\infty$, and $(q,q^{'})$ conjugate exponents, to be chosen later. Observe that if we apply \eqref{inclusions} for $u_0 \in H^{\frac{s \nu}{2}}(\G)$, $b=\frac{s \nu}{2}$, $a=0$, and $q_0=\frac{2Q}{Q-\nu s}$, then $\tilde{q}_0=2$, and we have 
\begin{equation}
\label{prop2.hold1}
\|u_0\|_{L^{q_{0}}(\G)} \lesssim \|\R^{\frac{s}{2}}u_0\|_{L^2(\G)}<\infty\,.
\end{equation}
Choosing $2q=q_0$ in \eqref{prop2.Hold} so that $q=\frac{Q}{Q-\nu s}$, we get $q^{'}=\frac{Q}{\nu s}$, so that 
\begin{equation}
\label{emb}
\|\sqrt{p}u_0\|\L \lesssim \|p\|_{L^{\frac{Q}{\nu s}}(\G)}\|\R^{\frac{s}{2}}u_0\|_{L^2(\G)}^{2}<\infty\,,
\end{equation}
and by \eqref{prop2.et} we can estimate
\begin{eqnarray}
\label{prop2.est.m}
\|\sqrt{p}(\cdot)u(t,\cdot)\|\L  & \leq & \|u_0\|_{H^{\frac{s \nu}{2}}(\G)}^{2}+\|\sqrt{p}u_0\|\L\nonumber\\
& \lesssim & \|u_0\|_{H^{\frac{s \nu}{2}}(\G)}^{2}+\|p\|_{L^{\frac{Q}{\nu s}}(\G)}\|u_0\|^{2}_{H^{\frac{s \nu}{2}}(\G)}\nonumber\\
& \leq & \left(1+\|p\|_{L^{\frac{Q}{\nu s}}(\G)}\right)\|u_0\|_{H^{\frac{s \nu}{2}}(\G)}^{2}\,,
\end{eqnarray}
 uniformly in $t \in [0,T]$.
Additionally, \eqref{prop2.et}, using the estimate \eqref{prop2.est.m}, implies
\begin{equation}
\label{prop2.other.est}
\|\R^{\frac{s}{2}}u(t,\cdot)\|\L \lesssim \left(1+\|p\|_{L^{\frac{Q}{\nu s}}(\G)}\right)\|u_0\|_{H^{\frac{s \nu}{2}}(\G)}^{2}\,.
\end{equation}
To show our claim \eqref{prop2.claim}, it suffices to show the desired estimate for the solution norm $\|u(t,\cdot)\|_{L^2(\G)}$. To this end, observe that by the Sobolev embeddings \eqref{inclusions} and  H\"older's inequality, using \eqref{emb} with $p$ instead of $\sqrt{p}$, 
and $\|p^2\|_{L^{\frac{Q}{\nu s}}(\G)}=\|p\|^2_{L^{\frac{2Q}{\nu s}}(\G)}$,
one obtains 
\[
\|p u(t,\cdot)\|\L \lesssim \|p\|^{2}_{L^{\frac{2Q}{\nu s}}}\|\R^{\frac{s}{2}}u(t,\cdot)\|\L\,,
\]
where the last combined with \eqref{prop2.other.est} yields
\begin{equation}\label{prop2.est.hold2}
\|pu(t,\cdot)\|\L \lesssim \|p\|^{2}_{L^{\frac{2Q}{\nu s}}(\G)} \left(1+\|p\|_{L^{\frac{Q}{\nu s}}(\G)}\right)\|u_0\|_{H^{\frac{s \nu}{2}}(\G)}^{2}\,.
\end{equation}
Finally, using arguments similar to those we developed in Proposition \ref{prop1.clas.schr}, together with the estimate \eqref{prop2.est.hold2} we get
\begin{eqnarray*}
\label{prop2.est.f}
\|u(t,\cdot)\|\L & \leq & \|u_0\|\L+\|p(\cdot)u(t,\cdot)\|\L\nonumber\\
& \lesssim & \|u_0\|_{H^{\frac{s \nu}{2}}(\G)}^{2}  \left\{\left(1+\|p\|^{2}_{L^{\frac{2Q}{\nu s}}(\G)}\right) \left(1+\|p\|_{L^{\frac{Q}{\nu s}}(\G)}\right)\right\}\,,
\end{eqnarray*}
 uniformly in $t \in [0,T]$. The uniqueness of $u$ is immediate by the estimate \eqref{prop2.claim}, and this finishes the proof of Proposition \ref{prop2}.
\end{proof}

\section{Existence and uniqueness of the very weak solution}
Proving the existence and the uniqueness of the very weak solution to the Cauchy problem \eqref{sch.eq} requires to assume that the potential $p$ and the initial data $u_0$ in \eqref{sch.eq} satisfy some moderateness properties. Regarding the potential $p$, we have in mind cases where $p$ is strongly singular; like for instance when $p=\delta$ or $p=\delta^2$. In the first case the moderate properties of $p$ follow by Proposition \ref{prop.mol}, while, in the second case, we understand $\delta^2$ as an approximating family or in the Colombeau sense.

\begin{definition}[Moderateness]
	\begin{enumerate}
		\item Let $X$ be a normed space of functions on $\G$.  A net of functions $(f_\epsilon)_\epsilon \in X$ is said to be \textit{$X$-moderate} if there exists $N \in \mathbb{N}$ such that \[
		\|f_\epsilon\|_{X}\lesssim \epsilon^{-N}\,,
		\]
		uniformly in $\epsilon \in (0,1]$.
		\item A net of functions $(u_\epsilon)_\epsilon$ in $C([0,T]; H^{\frac{s \nu}{2}}(\G))$ is said to be \textit{$C([0,T]; H^{\frac{s \nu}{2}}(\G))$-moderate} if there exists $N \in \mathbb{N}$ such that 
		\[
		\sup_{t \in [0,T]}\|u(t,\cdot)\|_{H^{\frac{s \nu}{2}}(\G)}\lesssim \epsilon^{-N}\,,
		\]
		uniformly in $\epsilon \in (0,1]$.
	\end{enumerate}
\end{definition}

\begin{definition}[Negligibility]
\label{defn.negl}
{Let $Y$ be a normed space of functions on $\G$.
Let $(f_\epsilon)_\epsilon $, $(\tilde{f}_\epsilon)_\epsilon$ be two nets. Then, the net $(f_\epsilon-\tilde{f}_\epsilon)_\epsilon$ is called $Y$-{\em negligible}, if the following condition is satisfied
    \begin{equation}\label{def.cond.negl}
    \|f_\epsilon-\tilde{f}_\epsilon\|_{Y}\lesssim \epsilon^k\,,
    \end{equation}
    for all $k \in \mathbb{N}$, $\epsilon \in (0,1]$. In the case where $f=f(t,x)$ is a function also depending on $t \in [0,T]$, then the {\em negligibility condition} \eqref{def.cond.negl} can be regarded as
    \[
    \|f_\epsilon(t,\cdot)-\tilde{f}_\epsilon(t,\cdot)\|_{Y}\lesssim \epsilon^k\,,\quad \forall k \in \mathbb{N}\,,
    \]
    uniformly in $t \in [0,T]$.} The constant in the inequality \eqref{def.cond.negl} can depend on $k$ but not on $\epsilon$.
\end{definition}

Definitions \ref{defn.vws.schr} and \ref{defn.uniq.schr} introduce the notion of the unique very weak solution to the Cauchy problem \eqref{sch.eq}. Our definitions resembles the ones in \cite{GR15}, but here we measure moderateness and negligibility in terms of $L^p(\G)$ or $H^{\frac{s \nu}{2}}(\G)$-norms rather than in terms of Gevrey-seminorms.
\begin{definition}[Very weak solution]\label{defn.vws.schr}
If there exists a $L^{\infty}(\G)$-moderate, or (provided that $Q> \nu s$) a $L^{\frac{2Q}{\nu s}}(\G)\cap L^{\frac{Q}{\nu s}}(\G)$-moderate {\em approximating net} $(p_\epsilon)_\epsilon$, $p_\epsilon\geq 0$ to $p$, and a $H^{\frac{s \nu}{2}}(\G)$-moderate regularising net $(u_{0,\epsilon})_\epsilon$ to $u_0$, then the net $(u_\epsilon)_\epsilon \in C([0,T]; H^{\frac{s \nu}{2}}(\G))$ which solves the $\epsilon$-parametrised problem 
	\begin{equation}\label{kg.reg}
		\begin{cases}
			i\partial_{t}u_\epsilon(t,x) +\mathcal{R}^{s}u_\epsilon (t,x)+p_\epsilon (x)u_\epsilon(t,x)=0\,,\quad (t,x)\in [0,T]\times \mathbb{G},\\
			u_\epsilon(0,x)=u_{0,\epsilon}(x)\,,\quad x \in \G\,,	
		\end{cases}       
	\end{equation}
for all $\epsilon \in (0,1]$, is said to be a \textit{very weak solution to the Cauchy problem \eqref{sch.eq}} if it is $H^{\frac{s \nu}{2}}(\G)$-moderate.
\end{definition}

\begin{remark}\label{rem.d^2}Let us mention that in Definition \ref{defn.vws.schr} the approximating net $p_\epsilon$ includes the case where $p_\epsilon$ is a \textit{regularisation} of $p$ in the case where $p \in \mathcal{D}^{'}(\G)$ is a distribution, i.e., for a Friedrichs mollifier $\psi\geq 0$ we define $p_\epsilon=p*\psi_\epsilon$. In singular cases, like for instance when $p=\delta^2$, we can think of $p_\epsilon$ as $p_\epsilon=\psi_{\epsilon}^{2}$; see also Remark \ref{rem.negl} for additional clarifications.
\end{remark}


Next we formulate the very weak existence result in compatibility with the two possible moderateness assumptions on the approximating nets $(p_\epsilon)_\epsilon$ as stated in Definition \ref{defn.vws.schr}. 
Before doing that, let us mention that, regarding the moderateness assumption of the regularisations (or appro\-ximations), the global structure of $\mathcal{E}^{'}$-distributions, implies that, for any regularisation of them taken via convolutions with a mollifier as in \eqref{mol}, the assumption on the $L^{p}$-moderateness, for $p \in [1,\infty]$, is natural. Formally we have the following proposition as in Proposition 4.8 in \cite{CRT21}.
 
{ \begin{proposition}\label{prop.mol}
Let $v \in \mathcal{E}^{'}(\G)$, and let 
$v_\epsilon=v*\psi_\epsilon$ be obtained as
the convolution of $v$ with a mollifier $\psi_\epsilon$ as in \eqref{mol}.
Then the regularising net $(v_\epsilon)_\epsilon$ is $L^{p}(\G)$-moderate for any $p \in [1,\infty]$.
\end{proposition}}
As an immediate consequence of Proposition \ref{prop.mol} is that, for the existence of the very weak solution to the Cauchy problem \eqref{sch.eq}, we do not require that the initial data $u_0$ is necessarily an element of the space $H^{\frac{s \nu}{2}}(\G)$ as Proposition \ref{prop1.clas.schr} and Proposition \ref{prop2} on the existence of the classical solution of \eqref{sch.eq} indicate. Indeed, we also allow that $u_0 \in \mathcal{E}^{'}(\G)$ is compactly supported distribution.
\begin{theorem}
	\label{thm.ex.kg}
	Let $u_0 \in H^{\frac{s \nu}{2}}(\G) \cup \mathcal{E}^{'}(\G)$. 
Then the Cauchy problem \eqref{sch.eq} has a very weak solution.
\end{theorem}
\begin{proof}
Let $u_0$ be as in the hypothesis. If $(p_\epsilon)_\epsilon$ is $L^{\infty}(\G)$-moderate (or $L^{\frac{2Q}{\nu s}}(\G)\cap L^{\frac{Q}{\nu s}}(\G)$-moderate) and $(u_{0,\epsilon})_\epsilon$ is $H^{\frac{s \nu}{2}}(\G)$-moderate, then, since also $p_\epsilon \geq 0$, by using \eqref{prop1.claim} (or \eqref{prop2.claim}, respectively) we get 
	\[
\|u_\epsilon(t,\cdot)\|_{H^{\frac{s \nu}{2}}(\G)} \lesssim \epsilon^{-N}\,,\quad N \in \mathbb{N}\,,
	\]
	for all $t \in [0,T]$ and for any $\epsilon \in (0,1]$. This means that the family of solutions $(u_\epsilon)_\epsilon$ is $H^{\frac{s \nu}{2}}(\G)$-moderate, and completes the proof of Theorem \ref{thm.ex.kg}.
\end{proof}

Roughly speaking, proving well-posedness in the very weak sense amount to proving that a very weak solution exists and it is unique modulo negligible nets. For the Cauchy problem \eqref{sch.eq} that we consider here, this notion can be formalised as follows.
\begin{definition}\label{defn.uniq.schr}
{Let $X$ and $Y$ be normed spaces of functions on $\G$.
We say that the Cauchy problem \eqref{sch.eq} has an $(X,Y)$-unique very weak solution, if for all $X$-moderate nets
	$p_\epsilon\geq 0, \tilde{p}_\epsilon\geq 0,$ such that $(p_\epsilon-\tilde{p}_\epsilon)_\epsilon$ is $Y$-negligible}, and for all $H^{\frac{s \nu}{2}}(\G)$-moderate regularisations $u_{0,\epsilon}$, $\tilde{u}_{0,\epsilon}$ such that $(u_{0,\epsilon}-\tilde{u}_{0,\epsilon})_\epsilon$ is $H^{\frac{s \nu}{2}}(\G)$-negligible, it follows that
	\[
	\|u_\epsilon(t,\cdot)-\tilde{u}_\epsilon(t,\cdot)\|_{L^2(\G)} \leq C_N \epsilon^N\,,\quad \forall N \in \mathbb{N}\,,
	\]
	 uniformly in $t \in [0,T]$, and for all $\epsilon \in (0,1]$, where $(u_\epsilon)_\epsilon$ and $(\tilde{u}_\epsilon)_\epsilon$ are the families of solutions corresponding to the $\epsilon$-parametrised problems 
	 \begin{equation}\label{kg.reg.tild}
		\begin{cases}
			i\partial_{t}u_\epsilon(t,x) +\mathcal{R}^{s}u_\epsilon (t,x)+p_\epsilon(x)u_\epsilon(t,x)=0\,,\quad (t,x)\in [0,T]\times \mathbb{G},\\
			u_\epsilon(0,x)=u_{0,\epsilon}(x),\; x \in \G\,,	
		\end{cases}       
	\end{equation}
	and 
\begin{equation}\label{kg.ret.notil}
		\begin{cases}
			i\partial_{t}\tilde{u}_\epsilon(t,x) +\mathcal{R}^{s}\tilde{u}_\epsilon (t,x)+\tilde{p}_\epsilon(x)\tilde{u}_\epsilon(t,x)=0\,,\quad (t,x)\in [0,T]\times \mathbb{G},\\
			\tilde{u}_\epsilon(0,x)=\tilde{u}_{0,\epsilon}(x),\; x \in \G\,,	
		\end{cases}       
	\end{equation}
	respectively.	
\end{definition}

\begin{remark}\label{rem.negl}
 Definition \ref{defn.uniq.schr} is a rigorous version of Definition 2.2 in the previous paper \cite{ARST21b} regarding the uniqueness of the very weak solution to the Cauchy problem \eqref{sch.eq} in the Euclidean setting. In particular,
 in \cite{ARTS21b} we assume that $p_\epsilon$ and $\tilde{p}_\epsilon$ are regularisations of $p \in \mathcal{D}^{'}(\G)$, and so they approximate $p$ is some suitable sense. Instead, in Definition \ref{defn.uniq.schr} we do not require the nets $p_\epsilon$, $\tilde{p}_\epsilon$ to approximate $p$; for instance, if $p_\epsilon$ is some regularisation of $p$ and $\tilde{p}_\epsilon$ is given as
    \begin{equation}\label{mex}
\tilde{p}_\epsilon=p_\epsilon+e^{-1/\epsilon}\,,
	\end{equation}
	then the net 
$(p_\epsilon-\tilde{p}_\epsilon)_\epsilon$ is $L^\infty$-negligible, and so satisfies the assumption described in Definition \ref{defn.uniq.schr}.
 Moreover, the absence of the approximation requirement, allows to consider singular cases of $p$; cf. Remark \ref{rem.d^2} where we take $p=\delta^2$ and $p_\epsilon=\psi_{\epsilon}^{2}$. Thus, under the choice of $\tilde{p}_\epsilon$ as in \eqref{mex}, the implied net $(p_\epsilon-\tilde{p}_\epsilon)$ is suitable for our purposes. 

To summarise the above, the meaning of the conditions regarding the nets $p_\epsilon$ and $\tilde{p}_\epsilon$ in Definition \ref{defn.uniq.schr} should be interpreted as a requirement for the stability of the very weak solution under negligible changes on the potential $p$; see also Theorem \ref{thm.uniq.kg.1} and Theorem \ref{thm.uniq.kg.2} where no approximating assumption has been regarded.
\end{remark}

The following theorems show the uniqueness of the very weak solution to the Cauchy problem \eqref{sch.eq} under different assumptions on the nets $(p_\epsilon)_\epsilon$. In order to do this, we need the following technical lemma, that shall also be used to prove the consistency of the very weak solution with the classical one.
\begin{lemma}
\label{techn.lem}
Let $u_0 \in L^2(\G)$ and assume that $p$ is non-negative. Then, for the unique solution $u$ to the Cauchy problem \eqref{sch.eq} we have the energy conservation
\begin{equation}\label{en.cons}
\|u(t,\cdot)\|_{L^2(\G)}=\textnormal{constant}\,,
\end{equation}
for all $t \in [0,T]$.
\end{lemma}
\begin{proof}
If we multiply equation \eqref{sch.eq} by $-i$, then we obtain 
\[
u_{t}(t,x)-i \R^s u(t,x)-ip(x)u(t,x)=0\,.
\]
If we multiply the above with $u$, integrate over $\G$, and consider the real part of the above we get 
\[
\Re(\langle u_t(t,\cdot),u(t,\cdot)\rangle_{L^2(\G)}-i\langle \R^su(t,\cdot),u(t,\cdot) \rangle_{L^2(\G)}-i\langle p(\cdot)u(t,\cdot),u(t,\cdot) \rangle_{L^2(\G)})=0\,,
\]
or equivalently
\[
\Re (\langle u_t(t,\cdot).u(t,\cdot)\rangle_{L^2(\G)})=\frac{1}{2}\partial_t\|u(t,\cdot)\|\L=0\,.
\]
The latter means that we have energy conservation, i.e., the norm $\|u(t,\cdot)\|_{L^2(\G)}$ remains constants over time, and in particular we have
\[
\|u(t,\cdot)\|_{L^2(\G)}=\|u_0\|_{L^2(\G)}\,,\quad \forall t \in [0,T]\,,
\]
implying \eqref{techn.lem}.
\end{proof}

\begin{theorem}\label{thm.uniq.kg.1}
	 Suppose that $u_0 \in H^{\frac{s \nu}{2}}(\G)\cup \mathcal{E}^{'}(\G)$.
	 {Then the very weak solution to the Cauchy problem \eqref{sch.eq} is
$(L^\infty(\G), L^\infty(\G))$-unique.}
\end{theorem}
\begin{proof}
Let  $(u_\epsilon)_\epsilon$ and $(\tilde{u}_\epsilon)_\epsilon$ be the families of solutions corresponding to the Cauchy problems \eqref{kg.reg.tild} and \eqref{kg.ret.notil}, respectively. If we denote by $U_\epsilon(t,\cdot):=u_\epsilon(t,\cdot)-\tilde{u}_\epsilon(t,\cdot)$, then $U_\epsilon$ satisfies
	\begin{equation}\label{kg.uniq}
		\begin{cases}
			i\partial_{t}U_\epsilon(t,x) +\mathcal{R}^{s}U_\epsilon (t,x)+p_\epsilon(x)U_\epsilon(t,x)=f_\epsilon(t,x)\,,\quad (t,x)\in [0,T]\times \mathbb{G},\\
			U_\epsilon(0,x)=(u_{0,\epsilon}-\tilde{u}_{0,\epsilon})(x)\,, x \in \G\,,	
		\end{cases}       
	\end{equation}
where $f_\epsilon(t,x):=(\tilde{p}_\epsilon(x)-p_\epsilon(x))\tilde{u}_\epsilon(t,x)$. \\

The solution of the Cauchy problem \eqref{kg.uniq} can be expressed in terms of the solution to the corresponding homogeneous Cauchy problem using Duhamel's principle. Indeed, if $W_\epsilon(t,x)$, and $V_\epsilon(t,x;\sigma)$, where $\sigma$ is some fixed parameter in $[0,T]$, are the solutions to the homogeneous Cauchy problems 
\begin{equation*}\label{hom.Ve}
\begin{cases*}
	i\partial_{t}V_\epsilon(t, x;\sigma)+\R^s V_\epsilon(t, x;\sigma)+p_\epsilon V_\epsilon(t, x;\sigma)=0\,,\quad\text{in}\, (\sigma,T] \times \G,\\
	V_\epsilon(t, x;\sigma)=f_\epsilon(\sigma,x)\,\quad \text{on}\,\{t=\sigma\}\times \G\,,
\end{cases*}
\end{equation*}
and 
\begin{equation*}
\begin{cases*}
	i\partial_{t}W_\epsilon(t, x)+\R^s W_\epsilon(t, x)+p_\epsilon W_\epsilon(t, x)=0\,,\quad\text{in}\, [0,T] \times \G,\\
	W_\epsilon(t, x)=(u_{0,\epsilon}-\tilde{u}_{0,\epsilon})(x)\,\quad \text{on}\,\{t=0\}\times \G\,,
\end{cases*}
\end{equation*}
respectively, then $U_\epsilon$ is given by 
\begin{equation}\label{duh.prin}
U_\epsilon(t,x)=W_\epsilon(t,x)+\int_{0}^{t}V_\epsilon(t-\sigma,x;\sigma)\,d\sigma\,.
\end{equation}
Taking the $L^2$-norm in \eqref{duh.prin} and using the energy conservation \eqref{en.cons} to estimate $V_\epsilon$ we get 
\begin{eqnarray}\label{mink}
\|U_\epsilon(t,\cdot)\|_{L^2(\G)} & \leq & \|W_\epsilon(t,\cdot)\|_{L^2(\G)}+\int_{0}^{T}\|V_\epsilon(t-\sigma,\cdot;\sigma)\|_{L^2(\G)}\,d\sigma\\
& \lesssim & \|u_{0,\epsilon}-\tilde{u}_{0,\epsilon}\|_{L^2(\G)}+\int_{0}^{T}\|f_\epsilon(\sigma,\cdot)\|_{L^2(\G)}\,d\sigma\nonumber\\
& \lesssim & \|u_{0,\epsilon}-\tilde{u}_{0,\epsilon}\|_{L^2(\G)}+\|\tilde{p}_\epsilon-p_\epsilon\|_{L^\infty(\G)}\int_{0}^{T}\|\tilde{u}_\epsilon(\sigma,\cdot)\|_{L^2(\G)}\,d\sigma\nonumber\,,
\end{eqnarray}
for all $t \in [0,T]$, where for the first inequality \eqref{mink} we have applied Minkowski's integral inequality, i.e., that 
\[
\|\int_{0}^{t}V_\epsilon(t-\sigma,\cdot;\sigma)\,d\sigma\|_{L^2(\G)}\leq \int_{0}^{t}\|V_\epsilon(t-\sigma,\cdot;\sigma)\|_{L^2(\G)}\,d\sigma\,.
\]

 Now, using the fact that $(u_{0,\epsilon}-\tilde{u}_{0,\epsilon})_\epsilon$ is $H^{\frac{s \nu}{2}(\G)}$-negligible, while also that the net $(\tilde{u}_\epsilon)_\epsilon$, as being a very weak solution to the Cauchy problem \eqref{kg.reg.tild}, is $H^{\frac{s \nu}{2}}(\G)$-moderate and that $(p_\epsilon-\tilde{p}_\epsilon)_\epsilon$ is $L^{\infty}$-negligible, we get that 
\[
\|U_\epsilon(t,\cdot)\|_{L^2(\G)} \lesssim \epsilon^{N}+\epsilon^{\tilde{N}}  \int_{0}^{T}\epsilon^{-N_1}\,d\sigma  \,,
\]
for some $N_1 \in \mathbb{N}$, and for all $N,\tilde{N} \in \mathbb{N}$, $\epsilon \in (0,1]$. That is, we have
\[
\|U_\epsilon(t,\cdot)\|_{L^2(\G)} \lesssim \epsilon^{k}\,,
\]
for all $k \in \mathbb{N}$, and the last shows that the net $(u_\epsilon)_\epsilon$ is the unique very weak solution to the Cauchy problem \eqref{sch.eq}.
\end{proof}
Alternative to Theorem \ref{thm.uniq.kg.1} conditions on the nets $(p_\epsilon)_\epsilon, (\tilde{p}_\epsilon)_\epsilon$ that guarantee the very weakly well-posedness of \eqref{sch.eq} are given in the following theorem.

\begin{theorem}\label{thm.uniq.kg.2}
Let $Q>\nu s$, and suppose that $u_0 \in H^{\frac{s \nu}{2}}(\G)$. Then, the very weak solution to the Cauchy problem \eqref{sch.eq} is $(L^\infty(\G),L^{\frac{2Q}{\nu s}}(\G))$-unique. Moreover, 
the very weak solution to the Cauchy problem \eqref{sch.eq} is also 
$(L^{\frac{2Q}{\nu s}}(\G)\cap L^{\frac{Q}{\nu s}}(\G),L^{\frac{2Q}{\nu s}}(\G))$-unique
and
$(L^{\frac{2Q}{\nu s}}(\G)\cap L^{\frac{Q}{\nu s}}(\G),L^\infty(\G))$-unique.
\end{theorem}
\begin{proof}
We will only prove the $(L^\infty(\G),L^{\frac{2Q}{\nu s}}(\G))$-uniqueness as the other two uniqueness statements are similar. Using arguments similar to those developed in Theorem \ref{thm.uniq.kg.1}, we arrive at \begin{multline*}
    \|U_\epsilon(t,\cdot)\|_{L^2(\G)} \lesssim \|u_{0,\epsilon}-\tilde{u}_{0,\epsilon}\|_{L^2(\G)}+\int_{0}^{T}\|f_\epsilon(\sigma,\cdot)\|_{L^2(\G)}\,d\sigma
    \\ =\|u_{0,\epsilon}-\tilde{u}_{0,\epsilon}\|_{L^2(\G)}+ \int_{0}^{T}\|(\tilde{p}_\epsilon-p_\epsilon)(\cdot)\tilde{u}_\epsilon(\sigma,\cdot)\|_{L^2(\G)}\,d \sigma\,.
 \end{multline*}
for all $t\in[0, T]$.
Additionally, by applying H\"older's inequality, together with the Sobolev embeddings \eqref{inclusions}, we have 
    \[
    \|(\tilde{p}_\epsilon-p_\epsilon)(\cdot)\tilde{u}_\epsilon(t,\cdot)\|_{L^2(\G)} \leq \|\tilde{p}_\epsilon-p_\epsilon\|_{L^{\frac{2Q}{\nu s}}(\G)}\|\R^{\frac{s}{2}}\tilde{u}_\epsilon(t,\cdot)\|_{L^2(\G)}\,,
    \]
    where since $(\tilde{u}_\epsilon)$, as being the very weak solution corresponding to the $L^{\infty}(\G)$-moderate net $(\tilde{p}_\epsilon)_\epsilon$, is $H^{\frac{s \nu}{2}}(\G)$-moderate, we have 
    \[
    \|\R^{\frac{s}{2}}\tilde{u}_\epsilon(t,\cdot)\|_{L^2(\G)} \lesssim \epsilon^{-N_1}\,,\quad \text{for some}\,N_1 \in \mathbb{N}\,.
    \]
    Summarising the above, and since
    \[
   \|u_{0,\epsilon}-\tilde{u}_{0,\epsilon}\|_{L^2(\G)}, \|\tilde{p}_\epsilon-p_\epsilon\|_{L^{\frac{2Q}{\nu s}}(\G)}\lesssim \epsilon^{N}\,,\quad \forall N \in \mathbb{N}\,,
    \]we obtain 
    \[
    \|U_\epsilon(t,\cdot)\|_{L^2(\G)}\lesssim \epsilon^k\,,\quad \forall k \in \mathbb{N}\,,
    \]
uniformly in $t$, and this finishes the proof of Theorem \ref{thm.uniq.kg.2}.
\end{proof}

\section{Consistency of the very weak solution with the classical one}

The next theorems stress the conditions, on the potential $p$ and on the initial data $u_0$, under which, the classical solution to the Cauchy problem \eqref{sch.eq}, as given in Proposition \ref{prop1.clas.schr} or Proposition \ref{prop2}, can be recaptured by its very weak solution. To avoid any possible misunderstanding, let us clarify by a `regularisation' of $p$ we mean the net arising via the convolution of $p$ with non-negative Friedrichs mollifiers as in \eqref{mol}.

\begin{theorem}\label{thm.consis.1}
 Let $Q > \nu s$. Consider the Cauchy problem \eqref{sch.eq}, and let $u_0 \in H^{\frac{s \nu}{2}}(\G)$. Assume also that $p \in L^{\frac{2Q}{\nu s}}(\G)\cap L^{\frac{Q}{\nu s}}(\G) $, $p\geq 0$, and that $(p_\epsilon)_\epsilon$, is a regularisation of the potential $p$. Then the regularised net $(u_\epsilon)_\epsilon$ converges, as $\epsilon \rightarrow 0$, in  $L^2(\G)$ to the classical solution $u$ given by Proposition \ref{prop2}.
\end{theorem}
\begin{proof}
	Let $u$ be the classical solution of \eqref{sch.eq} given by Proposition \ref{prop2}, and let $(u_\epsilon)$ be the very weak solution of the regularised analogue of it as in \eqref{kg.reg.tild}. If we denote by $W_\epsilon(t,x):=u(t,x)-u_\epsilon(t,x)$, then $W_\epsilon$ solves the auxiliary Cauchy problem
	\begin{equation}\label{eq.m=0}
	\begin{cases*}
		i\partial_{t}W_\epsilon(t,x)+\R^s W_\epsilon(t,x)+p_\epsilon(x)W_\epsilon(t,x)=\eta_\epsilon(t,x), \\
		W_\epsilon(0,x)=(u_0-u_{0,\epsilon})(x)\,,
	\end{cases*}
	\end{equation}
	where $\eta_\epsilon(t,x):=(p_\epsilon(x)-p(x))u(t,x)$.
Using Duhamel's principle and arguments similar to Theorem \ref{thm.uniq.kg.1} we get the estimates
\begin{eqnarray}
\label{cons.thm1.sch}
 \|W_\epsilon(t,\cdot)\|_{L^2(\G)} & \lesssim & \|u_{0}-u_{0,\epsilon}\|_{L^2(\G)}+\int_{0}^{T}\|\eta_\epsilon(\sigma,\cdot)\|_{L^2(\G)}\,d\sigma\nonumber\\
 & = & \|u_{0}-u_{0,\epsilon}\|_{L^2(\G)}+ \int_{0}^{T}\|(p_\epsilon-p)(\cdot)u(\sigma,\cdot)\|_{L^2(\G)}d \sigma\nonumber\\
 & \lesssim & \|u_{0}-u_{0,\epsilon}\|_{L^2(\G)}+\int_{0}^{T}\|p_\epsilon-p\|_{L^{\frac{2Q}{\nu s}}(\G)}\|\R^{\frac{s}{2}}u(\sigma,\cdot)\|_{L^2(\G)}d\sigma,
\end{eqnarray}
 where to get the last inequality we apply H\"older's inequality and the Sobolev embeddings \eqref{inclusions}. Now, since by Proposition \ref{prop2} we have $u\in H^{\frac{s \nu}{2}}(\G)$, while also $p \in L^{\frac{2Q}{\nu s}}(\G)$, $u_0 \in H^{\frac{s \nu }{2}}(\G)$, we get that 
\[
\|u_{0}-u_{0,\epsilon}\|_{L^2(\G)}\,,\|p_\epsilon-p\|_{L^{\frac{2Q}{\nu s}}(\G)}\|\R^{\frac{s}{2}}u(\sigma,\cdot)\|_{L^2(\G)} \rightarrow 0\,,
\]
as $\epsilon \rightarrow 0$, so that by \eqref{cons.thm1.sch} and Lebesgue's dominated convergence theorem we get

 \begin{equation}
     \label{duh.e.0}
      \|W_\epsilon(t,\cdot)\|_{L^2(\G)} \rightarrow 0\,,
 \end{equation}
 uniformly in $t \in [\sigma,T]$, where $\sigma \in [0,T]$, i.e., the very weak solution converges to the classical one in $L^2$, and this finishes the proof of Theorem \ref{thm.consis.1}.
\end{proof}

In the following theorem we denote by $C_{0}(\G)$ the space of continuous functions on $\G$ vanishing at infinity, that is, such that for every $\epsilon>0$ there exists a compact set $K$ outside of which we have $|f|<\epsilon$. Note that $C_0(\G)$ is a Banach space if endowed with the norm $\|\cdot\|_{L^{\infty}(\G)}$.

\begin{theorem}\label{thm.consis.2}
Consider the Cauchy problem \eqref{sch.eq}, and let $u_0 \in H^{\frac{s \nu}{2}}(\G)$. Assume also that $p \in C_0(\G)$, $p\geq 0$, and that $(p_\epsilon)_\epsilon$, $p_\epsilon\geq 0$, is a regularisation of the coefficient $p$. Then the regularised net $(u_\epsilon)_\epsilon$ converges, as $\epsilon \rightarrow 0$, in $L^2(\G)$ to the classical solution $u$ given by Proposition \ref{prop1.clas.schr}. 
\end{theorem}

Before giving the proof of Theorem \ref{thm.consis.2}, let us make the following observation: If $p \in C_0(\G)$, then $\|p_\epsilon\|_{L^\infty(\G)}\leq C<\infty$, uniformly in $\epsilon \in (0,1]$. 

\begin{proof}[Proof of Theorem \ref{thm.consis.2}]

First observe that for $p$, $(p_\epsilon)_\epsilon$ as in the hypothesis, we have $p_\epsilon \in L^{\infty}(\G)$ for each $\epsilon \in (0,1]$. Hence, if we denote by $W_\epsilon$ the solution to the problem \eqref{eq.m=0}, then, reasoning as we did in Theorem \ref{cons.thm1.sch}, we obtain 
\[
\|W_\epsilon(t,\cdot)\|_{L^2(\G)}\lesssim \|u_0-u_{0,\epsilon}\|_{L^2(\G)}+\int_{0}^{T}\|(p_\epsilon-p)(\cdot)u(\sigma,\cdot)\|_{L^2(\G)}\,d\sigma\,,
\]
uniformly in $t \in [0,T]$. Now, since
\[
\|(p_\epsilon-p)(\cdot)u(\sigma,\cdot)\|_{L^2(\G)} \leq \|p_\epsilon-p\|_{L^{\infty}(\G)}\|u(\sigma,\cdot)\|_{L^2(\G)}\,,
\]
while by Lemmas 3.1.58 and 3.1.59 in \cite{FR16} we have 
    \[
    \|p_\epsilon-p\|_{L^{\infty}(\G)} \rightarrow 0\,,\quad \text{as}\quad \epsilon \rightarrow 0\,,
    \]
    summarising the above we get 
    \begin{equation}
        \label{thm.cons2.infty}
         \|W_\epsilon(t,\cdot)\|_{L^2(\G)}\rightarrow 0\,,\quad \text{as}\quad \epsilon \rightarrow 0\,,
\end{equation}
    and this completes the proof  of Theorem \ref{thm.consis.2}.
\end{proof}

\begin{remark}\label{finrem}
In the consistency result \cite[Theorem 2.3]{ARST21b} the assumption on the potential $p$ is regarded as $p \in L^{\infty}(\mathbb R^d)$. However, this assumption is not a sufficient one; we should instead ask for $p$ to be in the subspace $C_0(\mathbb R^d)$ of $L^{\infty}(\mathbb R^d)$, as follows by Theorem \ref{thm.consis.2} in the particular case where $\G=\mathbb R^d$ and $\mathcal{R}=-\Delta$. 
\end{remark}

\end{document}